%% file: main.tex
\documentclass[twoside, 11pt]{article}

\usepackage{graphicx, subcaption}
\usepackage[margin=1in, footnotesep=0.5in]{geometry}
\usepackage[sort&compress, numbers]{natbib} 
\usepackage{amsfonts, amsmath, amssymb, amsthm}
\usepackage[shortlabels]{enumitem}
\usepackage{booktabs}
\usepackage{mathtools}
\mathtoolsset{showonlyrefs}

\usepackage{color}

\usepackage{hyperref}
\definecolor{customdarkred}{RGB}{150,0,0}
\definecolor{customdarkgreen}{RGB}{0,150,0}
\definecolor{customdarkblue}{RGB}{0,0,150}
\hypersetup{colorlinks=true, linkcolor=customdarkred, citecolor=customdarkgreen, urlcolor=customdarkblue}
\usepackage{url}
\usepackage[linesnumbered, ruled]{algorithm2e}
\usepackage{authblk}

\input notation.tex

\title{Theoretical Foundations of \\ Ordinal Spherical Multidimensional Scaling}
\author[1]{Elizabeth Coda} 
\author[1,2]{Ery Arias-Castro}
\affil[1]{\small Department of Mathematics, University of California, San Diego} 
\affil[2]{\small Halıcıoğlu Data Science Institute, University of California, San Diego}
\date{}

\begin{document}
\maketitle
\thispagestyle{empty}

\begin{abstract}  
There has been general interest in spherically constrained embeddings as data with an inherently circular or spherical structure arise in a number of applications. While many methods have been proposed on the metric side, little work has been done on the ordinal side in terms of methodology or theory. Here, we focus on the fundamental question of uniqueness of ordinal spherical MDS: Given a realizable setting in which underlying objects lie on the unit sphere and all that is known are their ordinal comparisons of the form `object $i$ is more similar to object $j$ than object $k$', is it possible to uniquely recover the original objects, up to an orthogonal transformation, in the large-sample limit? We answer affirmatively, both in the setting just described, as well as in the settings of spherical external unfolding (aka lateration) and spherical internal unfolding (aka preference mapping) in their ordinal variants.
\end{abstract}

\section{Introduction}
\label{sec:intro}

 Multidimensional scaling (MDS) refers to a set of methods that aim to embed objects into Euclidean space given possibly incomplete (dis)similarity data. In the metric setting, the dissimilarities are taken to have intrinsic meaning or value, whereas in the non-metric or ordinal setting, only comparisons between objects are available (i.e., object $i$ is more similar to object $j$ than object $k$). Alternatively, as a special case of this setting, in the non-metric setting only the relative magnitudes of the (dis)similarities are considered meaningful. Ordinal data are natural in the context of human subjects, and as much of the early work in the field of MDS took place in the context of psychometrics, ordinal MDS has been of central interest since the inception of MDS \cite{kruskal1964a,kruskal1964b,shepard1962a,shepard1962b}, described by \citet{borg2007modern} as ``a special case of MDS, and possibly the most important one in practice". 

The development of MDS methods in the metric and ordinal settings have been heavily intertwined. On the metric side with complete dissimilarity information, two of the most prominent methods for embedding include classical scaling (CS) and stress minimization. Though commonly used for metric scaling, Kruskal originally proposed using the stress in the context of non-metric scaling \citep{kruskal1964a, kruskal1964b}. The subsequent literature on stress minimization and other MDS methods is vast, and we do not attempt to summarize it here, but we do highlight some of the parallels between methods on both the metric and ordinal side including majorization approaches for stress minimization (e.g., on the metric side \cite{smacof2} and ordinal side \cite{terada2014local}), as well as semidefinite programming (SDP) based algorithms (e.g., on the metric side \cite{alfakih1999solving, biswas2004semidefinite, weinberger2005nonlinear} and ordinal side \cite{agarwal2007generalized, tamuz2011adaptively, van2012stochastic}).

In the realizable case of the metric setting the dissimilarities correspond to Euclidean distances between objects in Euclidean space. In this setting, all configurations of points that satisfy the dissimilarity constraints are rigid transformations of one another. In the ordinal realizable setting, the ordinal information is congruent with Euclidean distances between objects in Euclidean space. Here, one can only hope to recover the objects up to similarity information, and only in the large-number-of-objects limit. Even though ordinal embedding methods have been used since the beginnings of MDS, formal theoretical results for ordinal MDS were established only somewhat recently. Indeed, while \citet{shepard1966} discussed the question in a continuum limit of the problem, more rigorous results were only established much later in  \citep{kleindessner2014, arias2017some}, where the uniqueness of the ordinal embedding solution up to similarity transformations was proved in the MDS setting. This was later extended in \cite{arias2025ordinal} to the settings of external unfolding (aka lateration) and internal unfolding (aka preference mapping). For the reader not already familiar with this terminology (borrowed from psychometrics), definitions are forthcoming.

However, while in traditional MDS the embeddings are unconstrained in Euclidean space, in a number of applications an embedding into a different space may be more appropriate. In particular, there has been notable interest in embedding into the circle or a higher-dimension sphere, with natural applications arising in psychometrics \citep{borg1980}, geospatial science \citep{fernando2025clustering}, environmental science \citep{fernando2025clustering}, astronomy \citep{zoubouloglou2023scaled}, biology \citep{zoubouloglou2023scaled, evangelopoulos2020circular, luo2024spherical, ding2021deep}, computer vision \citep{elad2005texture, wilson2010spherical}, and natural language processing \citep{meng2019spherical}. Yet methods for spherically constrained embeddings have almost exclusively focused on the metric side \cite{cox1991multidimensional, borg1980,  elad2005texture, miller2022spherical, perry2020drawing, fernando2025clustering, wilson2010spherical, wilson2014spherical,rebonato2011most, papazoglou2017examination, pietersz2004rank}, despite natural applications on the ordinal side, as in the analysis of gene expression data \cite{taguchi2005relational} or human preference data in \cite[Section III]{borg2007modern}.

Given the intertwined development of metric and ordinal MDS in the Euclidean setting, as well as the general interest in spherical MDS, it is natural to consider the setting of ordinal spherical MDS. Here, we focus not on methodology for ordinal spherical MDS, but rather, take a step back and establish some fundamental theory for ordinal spherical MDS. Following work in the Euclidean ordinal MDS \citep{kleindessner2014, arias2017some, arias2025ordinal}, we consider the question of uniqueness in the realizable setting. The central question of interest is: {\em When ordinal comparisons are congruent with cosine similarity between points on the unit sphere, is enough information present in the ordinal constraints to uniquely recover the unknown points in the large-sample limit, up to orthogonal transformations?} We answer this question affirmatively in the setting of complete triplet information. We also consider the ordinal variants of spherical external unfolding and spherical internal unfolding, and show that, in these settings too, the ordinal information is enough to recover the individuals and objects, up to the natural invariances of the problem.

\subsection{Related Work}
\label{sec:related_work}

One popular approach to obtaining a constrained embedding is constrained stress minimization. Early work included adding linear constraints \citep{de1980multidimensional} and constrained or confirmatory monotone distance analysis (CMDA), in which a penalty term is added to the stress function to weakly enforce constraints \citep{borg1980}. See also \citep{heiser1983, smacof2, borg2007modern} for more details on these methods. 

As an alternate approach to embedding on the sphere, \citet{cox1991multidimensional} propose stress minimization with a spherically parametrized embedding. In their stress formulation, distances between embedded points are measured using Euclidean distances, though others have also measured distances along the geodesics \cite{elad2005texture, miller2022spherical, perry2020drawing, fernando2025clustering, wilson2010spherical}. Similarly, using geodesic distances has been proposed to construct embeddings on quadratic surfaces \citep{smacof2} or on the hyperbolic plane \citep{walter2002interactive}. For these particular surfaces, analytical forms of the geodesic distances exist, however, practical difficulties arise when working with geodesics on ellipsoids \cite{smacof2}. \citet{bronstein2006} propose a stress minimization method for embedding into a general Riemannian surface by approximating geodesic distances via the fast marching method. While stress minimization was proposed in the context of non-metric scaling \citep{kruskal1964a, kruskal1964b}, with the exception of CMDA, these methods were proposed in the metric setting. 

Other approaches for spherically constrained embeddings are based on CS, which finds the best rank-$p$ approximation to the doubly centered matrix of squared input dissimilarities by a spectral decomposition. Alternatively, CS can be viewed as a method that minimizes the strain. Methods for spherical embeddings include projecting the embedding obtained via CS onto the sphere \citep{rebonato2011most, papazoglou2017examination}, a majorization approach to minimizing the strain with spherical constraints \citep{pietersz2004rank}, and a modified spectral approach \citep{wilson2014spherical}. More generally, \citet{deng2024neuc} propose a generalized version of CS for embedding into a non-Euclidean space. Closely related is also a line of work that seeks to find a low rank approximation to a correlation matrix, where the unit diagonal constraint yields a spherically constrained embedding \cite{ pietersz2004rank, falissard1996spherical, wu2002fast, zhang2003optimal, higham2002computing}. The problem of finding a spherically constrained Euclidean distance matrix has also been considered in the optimization literature \cite{bai2015constrained}. 

We also highlight recent work on spherical PCA \cite{li2022manifold, liu2019spherical, tabaghi2024principal, luo2024spherical}.  In this dimensionality reduction setting, the data are points in some (often high) dimensional Euclidean space, and in this particular variant of the problem, a representation into a low-dimensional sphere is sought out. This can be viewed as a special case of spherical MDS, as the dissimilarities can be obtained from the data by, e.g., computing the pairwise Euclidean distances. 

On the ordinal side, methods for constrained embedding are scarce. One approach is to convert ordinal data  into metric data, as is done for example in \citep{borg1980, smacof2} where preference data is first converted to dissimilarity data before embedding objects onto the sphere via an internal unfolding method. Other work includes embedding objects into hyperbolic space given ordinal constraints \citep{suzuki2019hyperbolic}, and generating sentence embeddings that lie on certain manifolds via a triplet loss \citep{chavan2025manifold}. We are not aware of any theoretical results for ordinal embedding into a sphere, or any other curved space for that matter, except for the article \citep{suzuki2019hyperbolic}, which are of a very different nature. 

\subsection{Outline of the paper}
We closely follow the structure of \citep{arias2025ordinal}. We begin with the simplest setting of external unfolding in \secref{external}, consider the setting of spherical MDS with complete triplet similarity comparisons in \secref{mds}, and then consider internal unfolding in \secref{internal}. In each section, we begin by stating the problem in the discrete setting, then consider the problem in a continuum setting as originally proposed by \citet{shepard1966}, and finally, leverage these results in a discrete asymptotic setting inspired by \citet{kleindessner2014} to establish asymptotic uniqueness (up to similarity transformations). We emphasize that while we take a similar approach to that in \cite{arias2025ordinal}, the results established there for the Euclidean space do not directly imply the results obtained here for the spherical space.

\subsection{Background and notation}
Throughout, the embedding space is the unit sphere in dimension $p$, that is,
\[\bbS^{p-1} := \{ x \in \bbR^p : \|x\| = 1\},\] 
where $\| \cdot \|$ denotes the Euclidean norm, with associated inner product denoted $\langle \cdot, \cdot \rangle$. 
We use $x_i$ for the $i$-th coordinate of the point or vector $x$, so that $x = (x_1, \dots, x_p)$. The standard basis vectors of $\bbR^p$ are  $e_1, \dots, e_p$. For a positive integer $n$, $[n] := \{1, \dots, n\}$.

%We now review some basic concepts related to the topology and geometry of the sphere, and in doing so, also introduce some additional notation. 
A set $V \subset \bbS^{p-1}$ is open in $\bbS^{p-1}$ if there exists an open set $U \subset \bbR^p$ such that $V = U \cap \bbS^{p-1}$. An example of that is the open ball in the sphere centered at $x \in \bbS^{p-1}$ and of radius $\eps$, defined as 
$$\sball(x,\eps) = \ball(x,\eps) \cap \bbS^{p-1},$$ 
where $\ball(x,\eps)$ denotes the open ball in $\bbR^p$ centered at $x$ and of radius $\eps$. 
The interior of a set $U \subset \bbS^{p-1}$ relative to $\bbS^{p-1}$ is defined as 
$$\sint(U) = \Big\{x \in U : \text{there exists } \eps >0 \text{ such that } \sball(x,\eps) \subset U \Big\}.$$
Note that the closure of a set $U \subset \bbS^{p-1}$ relative to $\bbS^{p-1}$ coincides with its closure relative to the surrounding Euclidean space $\bbR^p$.

For any pair of points $x,x' \in \bbS^{p-1}$ not diametrically opposed, there is a unique great circle passing through both points, which can be parametrized by arc length via an angle $\theta$ as $\tilde{\gamma}_{(x,x')}: [0,\theta_{(x,x')}] \to \bbS^{p-1}$, where $\theta_{(x,x')} = \cos^{-1}(\langle x, x'\rangle)$ is the angle between $x$ and $x'$ seen as vectors, and 
\begin{align}
\label{eq:great_circle}
    \tilde{\gamma}_{(x,x')}(\theta) =  \cos(\theta) x + \sin(\theta) \frac{x' - \langle x, x'\rangle x}{\|x' - \langle x, x'\rangle x\|}.
\end{align} 
We will often use the parameterization $\gamma_{(x,x')}: [0,1] \to \bbS^{p-1}$, where $\gamma_{(x,x')}(t) :=  \tilde{\gamma}_{(x,x')}(t \theta_{(x,x')})$, which can be extended to the entire real line. The arc itself as a curve on $\S$ will be denoted $\arc(x,x') := \gamma_{(x,x')}([0,1])$. For $x,x' \in \bbS^{p-1}$ not diametrically opposed, their spherical midpoint is defined as the unique equidistant point on the shortest arc between them. This point is $\gamma_{(x,x')}(1/2)$, and can also be obtained by normalizing their Euclidean midpoint, i.e., as $(x+x')/\|x+x'\|$. 

% We say $A \subset \bbS^{p-1}$ is spherically convex if for every $x,x' \in A$ not diametrically opposed, $\gamma_{(x,x')}(t) \in A$ for all $t \in [0,1]$. The spherical convex hull of a set $A \subset \bbS^{p-1}$, denoted $\conv(A)$, is the smallest spherically convex set containing $A$. Finally, we say a set $A \subset S$ is spherically symmetric about $x$ if for any $x' \in A$, $\gamma_{(x, x')}(-1) \in A$, that is, there is $x'' \in A$ such that $x$ is the spherical midpoint of $x'$ and $x''$. For conciseness, we will often say convex, convex hull, and symmetric in place of spherically convex, spherical convex hull, and spherically symmetric when the meaning is clear from the context. 

For two distinct points $x,x' \in \bbR^p$, define $\hplane(x,x') = \{y : \|y-x\| = \|y-x'\|\}$, which is the affine hyperplane passing through $\frac12(x+x')$ perpendicular to $x-x'$; we also define $\hplane^+(x,x') = \{y : \|y-x\| < \|y-x'\|\}$, which is one of the two open half-spaces defined by that hyperplane. Note that, if $x,x' \in \S$, $\hplane(x,x')$ is a linear hyperplane as it passes through the origin, and it also passes through the midpoint of $x,x'$ on the sphere. 
Unless otherwise specified, a hyperplane is linear. 
Hyperplanes play a central role because of the following elementary result (whose proof is left to the reader).
\begin{lemma}\label{lem:openhyper}
Suppose $\cY$ is an open subset of $\bbS^{p-1}$. If $\cH$ is a hyperplane that intersects $\cY$, i.e., $\cH \cap \cY \ne \varnothing$, then $\cH = \hplane(y,y')$ for some $y,y' \in \cY$.
\end{lemma}

The following result regarding equiangular points on the sphere will also be useful.

\begin{lemma}
\label{lem:z_config}
    Let $z^{\pm}_1, \dots, z^{\pm}_p$ be distinct points on $\bbS^{p-1}$ such that \begin{equation}
    \label{eq:z_constraint}
        \langle z_i^s, z_j^t \rangle = \langle z_k^u, z_l^v \rangle
\quad \text{for all } i \neq j,\ k \neq l,\ \text{and } s,t,u,v \in \{+,-\}.
\end{equation}
    Then, up to an orthogonal transformation, the only possible configuration is $z_i^{\pm} = \pm e_i$.
%    , where $e_i$ are the standard basis vectors of $\bbR^p$. 
\end{lemma}

\begin{proof}
    We proceed by induction. In the case $p=2$, without loss of generality, we may assume $z_1^+ = e_1$. Then, to satisfy $\langle z_1^+, z_2^+ \rangle =  \langle z_1^+, z_2^- \rangle$, we may, after possible relabeling, take $z_2^{\pm} = \cos(\theta)e_1 \pm \sin(\theta)e_2$ for some $\theta > 0$. However, it must also be the case that $\langle z_1^-, z_2^+ \rangle =  \langle z_1^-, z_2^- \rangle$, implying that $z_1^-$ is orthogonal to $z_2^+ - z_2^- = 2\sin(\theta)e_2$. Since $z_1^- \in \bbS^{1}$ must be distinct from $z_1^+$, the only possibility is $z_1^- = -e_1$.
    Finally, to satisfy $\langle z_1^+, z_2^+ \rangle =  \langle z_1^-, z_2^+ \rangle$, we must have $\theta = \pi/2$, concluding that $z_2^\pm = \pm e_2$.

    Now, for general $p$, assume $z_1^+ = e_1$. In order to satisfy $\langle z_1^+, z_2^+ \rangle = \langle z_1^+, z_2^- \rangle = \dots = \langle z_1^+, z_p^+ \rangle = \langle z_1^+, z_p^- \rangle $, for $2 \leq i \leq p$, we can take $z_i^{\pm} = \cos(\theta)e_1 + \sin(\theta)u_i^\pm$ for some $u_i^{\pm} \in \bbS^{p-2}$. The constraint in \eqref{eq:z_constraint} also implies that
    $$
     \langle u_i^s, u_j^t \rangle = \langle u_k^u, u_l^v \rangle
\quad \text{for all } i \neq j,\ k \neq l,\ \text{and } s,t,u,v \in \{+,-\}.$$
    By induction, up to an orthogonal transformation, we must have $u_i^\pm = \pm v_{i-1}$, where $v_i$ are the standard basis vectors of $\bbR^{p-1}$. Then, in order to satisfy $\langle z_1^+, z_2^+ \rangle = \langle z_2^+, z_3^+ \rangle$, $\theta$ must be $=\pi/2$. We have thus shown that, up to an orthogonal transformation, $z_1^+ = e_1$, and $z_i^{\pm} = \pm e_i$ for $2 \leq i \leq p$. It remains to place $z_1^-$, but it can readily be seen that the only possibly that satisfies the constraints is $z_1^- = -e_1$. 
\end{proof}

\section{Spherical external unfolding}
\label{sec:external}

In Euclidean external unfolding, the position of several objects is assumed known (possibly via some MDS method), and the task is to locate one or more individuals in space based on ranking or preference data for these objects. The term ``external unfolding" is popular in the psychometrics literature, though the problem has also been considered in the engineering literature under the name ``lateration", where the objects are often referred to as ``landmarks" or ``anchors". Even when all the (dis)similarity information is available, lateration can also be used for reasons of computational efficiency \citep{anderton2019, de2004sparse}. In spherical external unfolding, we will assume the position of several objects on the unit sphere is known, and the task is to locate one or more individuals on the unit sphere based on their preferences for the objects. This is the setting in \citep{yu2008}, for example. We focus on the ordinal variant of the problem.

\subsection{Discrete setting}
\label{sec:external_discrete}

\begin{figure}
    \centering
    \includegraphics[width=0.5\linewidth]{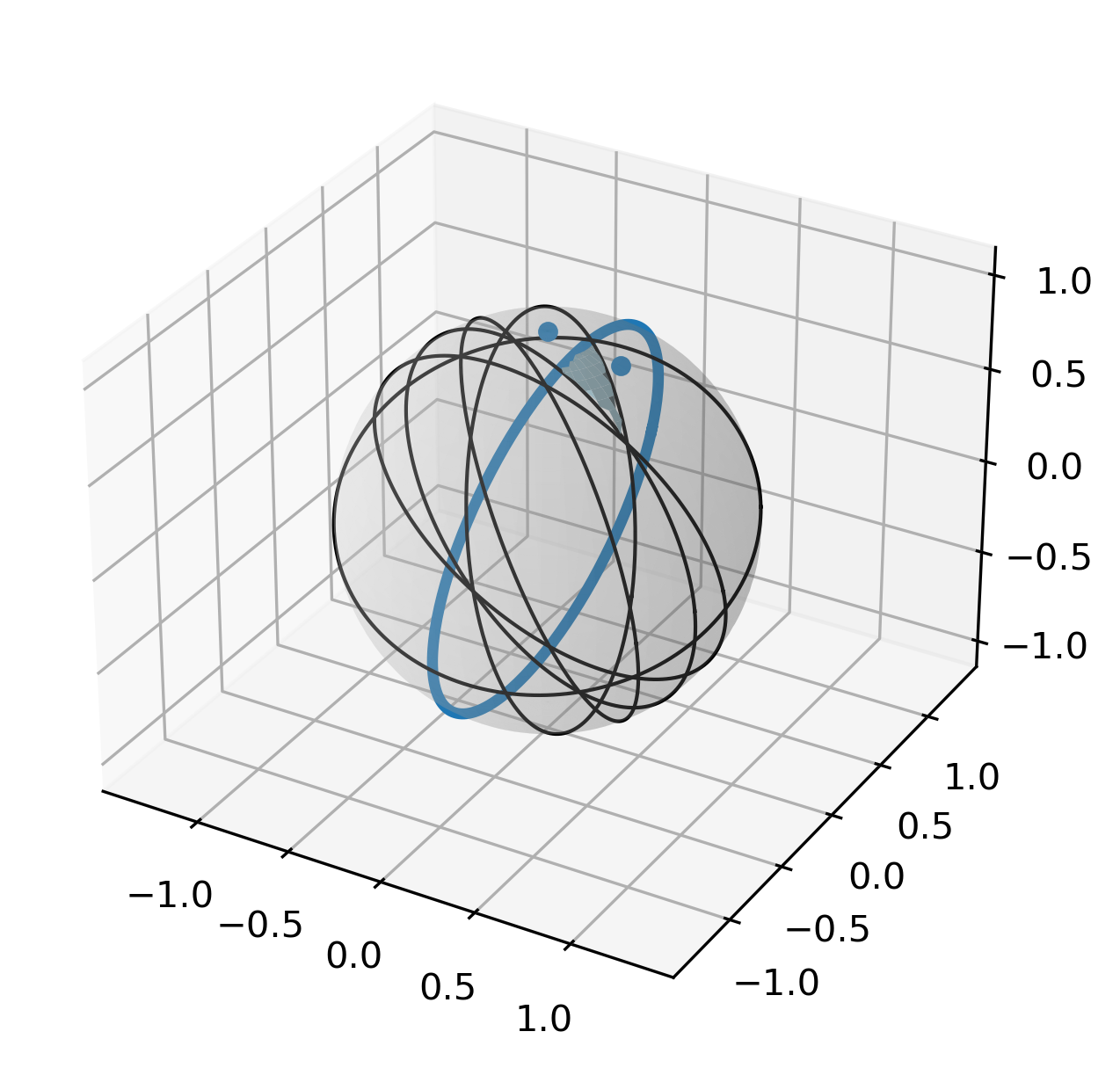}
    \caption{An illustration of the spherical external unfolding problem. The information that individual $x$ prefers one highlighted object over the other $x$ to a hemisphere. With an increasing number of such comparisons, the localization of $x$ is ever more precise. (Figure inspired by \citep{massimino2021}.)}
\end{figure}

With spherical constraints, the problem encountered in practice is as follows: Given landmark points $\cY_n := \{y_1, \dots, y_n\} \subset \bbS^{p-1}$ and a permutation $(r_1, \dots, r_n)$ of $(1, \dots, n)$, and a dimension $p \geq 2$, \begin{equation}
\label{eq:external_problem}
    \text{Find } x \in \bbS^{p-1} \text{ such that } \langle x,y_i\rangle < \langle x,y_j\rangle \text{ whenever } r_i < r_j.
\end{equation}

In the realizable setting, there is some unknown point, $\hat{x} \in \bbS^{p-1}$, such that $r_i < r_j$ whenever $ \langle \hat{x}, y_i \rangle < \langle \hat{x}, y_j \rangle$. In the finite case, the solutions to \eqref{eq:external_problem} form an open set, so it is not possible to recover an individual $\hat{x}$ exactly from its preferences for objects in $\cY_n$. However, as the number of landmark objects $n \to \infty$, we expect the solution set to shrink to a singleton set --- at least under some conditions on the landmark set $\cY_n$. 

In the Euclidean setting, it was shown in \citep{arias2025ordinal} that  as the number of landmark objects becomes dense in some (non-empty) open set, it is possible to recover the unknown point based on only ordinal distance comparisons. In the spherical setting, the landmark objects are restricted to the sphere, so that the result in \citep{arias2025ordinal} does not directly apply.

%\begin{remark}
%    In view of the fact that $\|x-y\|^2 = 2(1-\langle x , y \rangle) $ for all $x,y \in \bbS^{p-1}$, the problem in \eqref{eq:external_problem} is the same as that in \cite[Sec~2]{arias2025ordinal}, but with the embedding space being $\bbS^{p-1}$ instead of $\bbR^p$. 
%\end{remark}

\begin{remark}
We have assumed that both the individual and the objects are on the sphere. In another version of spherical external unfolding, one may wish to embed individuals on the sphere based on their preferences for objects which are not necessarily confined to the sphere. The results derived in this section, in particular the main result (\thmref{external_asymptotic}), can be adapted to work to within this other version of spherical external unfolding with only minor modification in both statement and proof. 
\end{remark}

\subsection{Continuum setting} 
Before considering the discrete set $\cY_n$, we turn to the continuum setting originally proposed by \citet{shepard1966}. For that, we consider a continuously infinite set of landmark points $\cY$. The set $\cY$ can be thought of as a ``filled-in" version of the set $\cY_n$ in the large-number-of-landmarks limit $n \to \infty$.  We are interested in conditions on $\cY$ under which it is possible to localize a point $x \in \bbS^{p-1}$ given all its preferences for objects in $\cY$. We reformulate the problem using the notion of equivalent points, where $x,x' \in \bbS^{p-1}$ are equivalent with respect to $\mathcal{Y}$ if 
$$
\langle x,y\rangle < \langle x,y'\rangle \iff\langle x',y\rangle < \langle x',y'\rangle, \qquad \text{for all } y,y' \in \mathcal{Y}.
$$
In other words, two points are equivalent with respect to $\cY$ if they are on the same side of any hyperplane defined by pairs of points belonging to $\cY$. Equivalent points are indistinguishable with respect to their preferences for objects in $\cY$. Thus, if equivalent points with respect to $\cY$ must coincide, it is possible to uniquely determine an individual based only on ordinal preferences for objects in $\cY$. In fact, it will be enough to consider a weaker notion of equivalence, where we say $x'$ is weakly equivalent to $x$ with respect to $\mathcal{Y}$ if 
\begin{align}
\label{eq:ext_weakly_equiv}
\langle x,y\rangle < \langle x,y'\rangle \implies\langle x',y\rangle \leq \langle x',y'\rangle, \qquad \text{for all } y,y' \in \mathcal{Y}.    
\end{align}
In other words, $x'$ must be on the same side as $x$ of any hyperplane defined by pairs of points belonging to $\cY$ not passing through $x'$.

As we are constrained to the sphere, we only need to determine the direction of $x$ and not its magnitude. In \thmref{external_open}, we show that when $p \geq 3$, it is sufficient for $\cY$ to have non-empty interior. But when $p=2$, i.e. when embedding into the unit circle, we ask that $\cY$ contain an open hemisphere. In \figref{external_p=2}, we provide a counter-example showing that this condition cannot be weakened. 

\begin{figure}
    \centering
    \includegraphics[width=0.3\linewidth]{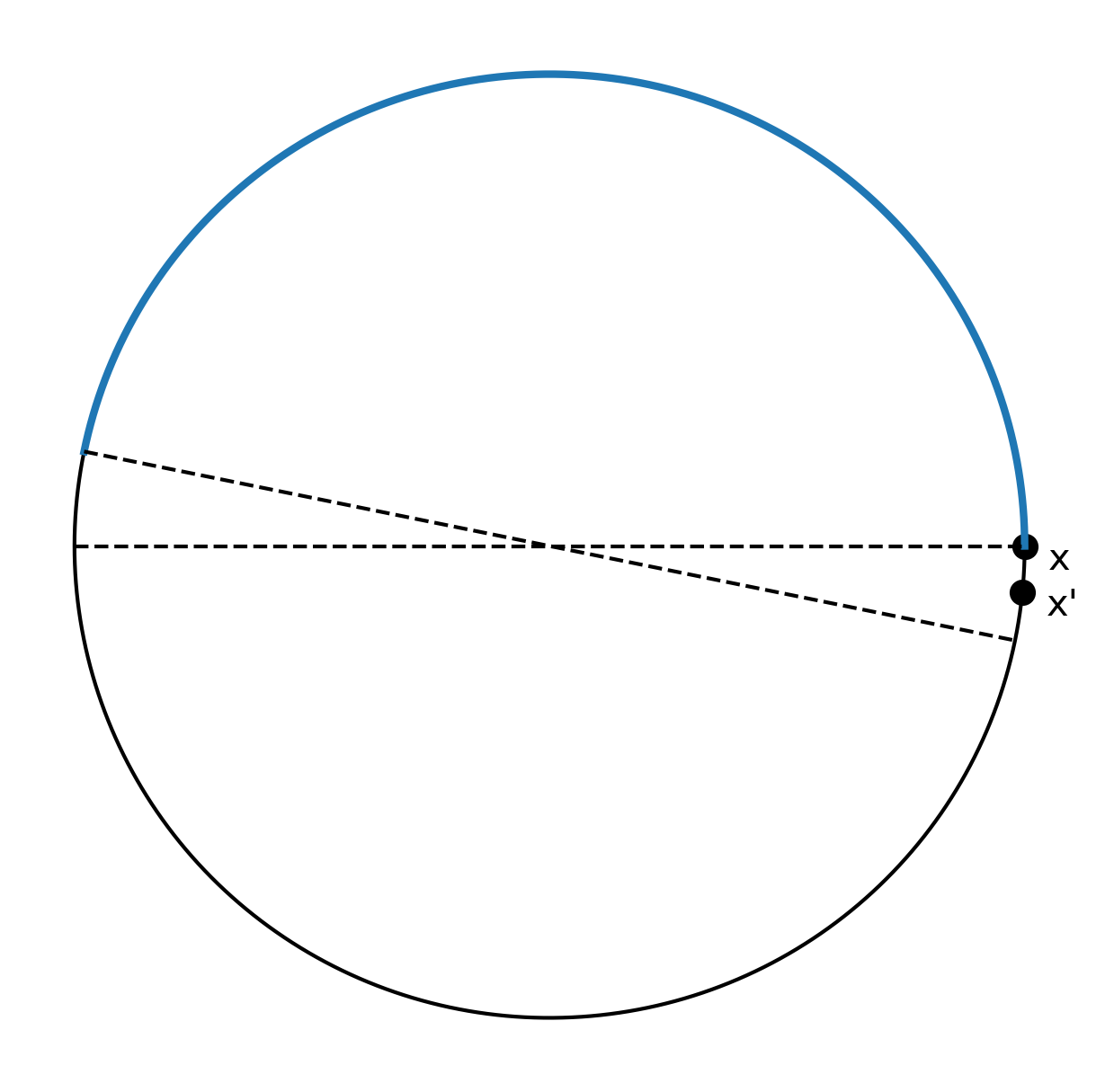}
    \caption{When $p=2$, $x$ and $x'$ can be weakly equivalent with respect to $\cY$ without coinciding. Here, $\cY = \{(\cos \theta, \sin \theta): 0 \leq \theta \leq \pi - 2\delta \}$ for some $\delta > 0$, and $x = (1,0)$, and $x' = (\cos \delta, -\sin \delta)$ are weakly equivalent, and yet distinct.}
    \label{fig:external_p=2}
\end{figure}

\begin{theorem}
\label{thm:external_halfsphere}
    If $\cY$ contains an open hemisphere, two points that are weakly equivalent with respect to $\cY$ in the sense of \eqref{eq:ext_weakly_equiv} must coincide. 
\end{theorem}

\begin{proof}
Suppose $x,x' \in \S$ are distinct. We need to show that they are not weakly equivalent with respect to $\cY$. 
Take two hyperplanes, say $\cH$ and $\cH'$,  intersecting $\arc(x,x')$ but not passing through either $x$ or $x'$. Then, necessarily, at least one of them intersects $\cY$, say $\cH$. 
By \lemref{openhyper}, there are $y,y' \in \cY$ such that $\hplane(y,y') = \cH$. Because $x,x'$ are on opposite sides of $\cH$, we may assume that, say, $x \in \hplane^+(y,y')$ while $x' \in \hplane^+(y',y)$, which then precludes $x,x'$ from being weakly equivalent with respect to $\cY$.
\end{proof}

\begin{theorem}
\label{thm:external_open}
    Assume $p \geq 3$. If $\cY$ has non-empty interior, two points that are weakly equivalent with respect to $\cY$ in the sense of \eqref{eq:ext_weakly_equiv} must coincide. 
\end{theorem}

\begin{proof}
Suppose $x,x' \in \S$ are distinct. We need to show that they are not weakly equivalent with respect to $\cY$. 
Consider all the hyperplanes passing through the midpoint of $x,x'$. As their union is the entire ambient space $\bbR^p$ (because $p \ge 3$), there must be one of them that intersects $\cY$. 
We then conclude exactly as we did in the proof of \thmref{external_halfsphere}.
\end{proof}

\subsection{Discrete asymptotic setting} 

We now return to the discrete setting in \secref{external_discrete} with landmark points $\cY_n = \{y_1, \dots, y_n \}$. We consider the large-number-of-landmarks limit $n \to \infty$ and let  $\cY_{\infty} := \{y_i : i \geq 1\}$. Guided by \thmref{external_halfsphere} and \thmref{external_open}, we assume $\cY_{\infty}$ is dense in some appropriate set of $\bbS^{p-1}$. In this asymptotic setting, when the situation is realizable, we establish that there is a unique individual on the sphere with the given preferences for objects in $\cY_{\infty}$.

\begin{theorem}
\label{thm:external_asymptotic}
If $p = 2$, assume that $\cl(\cY_{\infty})$ contains an open hemisphere, while if $p \geq 3$, assume that $\cl(\cY_{\infty})$ has non-empty interior. For each $n$, let $(x_n, x'_n) \in \bbS^{p-1} \times \bbS^{p-1}$ be a pair of points that are weakly equivalent with respect to $\mathcal{Y}_n$ in the sense of \eqref{eq:ext_weakly_equiv}. Then any accumulation point of $\{ (x_n, x'_n) : n \geq 1\}$ must be of the form $(x, x)$. In particular, if both $(x_n)$ and $(x'_n)$ converge, their limits must coincide.
\end{theorem}

\begin{proof}
We may assume without loss of generality that $x_n \to x$ and $x'_n \to x'$ as $n\to\infty$. We need to show that $x = x'$. 
Let $\cY = \sint(\cl(\cY_{\infty}))$. 
If $p=2$, \thmref{external_halfsphere} applies, while if $p \ge 3$, \thmref{external_open} applies, so that, either way, it is enough to show that $x,x'$ are weakly equivalent with respect to $\cY$. 

For that, take any $y,y' \in \cY$ such that $\langle x,y\rangle < \langle x,y'\rangle $. By continuity of the inner product, it must be the case that $\langle x_n,y\rangle < \langle x_n,y'\rangle$ for all $n$ large enough, in turn implying that $\langle x'_n,y\rangle \le \langle x'_n,y'\rangle$, by the assumption that $x_n,x'_n$ are weakly equivalent with respect to $\mathcal{Y}_n$ for all $n$. Taking the $n\to\infty$ limit, this forces $\langle x',y\rangle \le \langle x',y'\rangle$, again by continuity of the inner product.
\end{proof}

\section{Spherical multidimensional scaling}
\label{sec:mds}

In the setting of spherical multidimensional scaling, we have $n$ objects and the goal is to embed them on the unit sphere given pairwise (dis)similarity information. Most of the existing methods for spherical embedding highlighted in \secref{related_work} assume complete information in the metric setting. We focus on the ordinal setting. The most common scenarios are when all quadruple comparisons are available (object $i$ is more similar to object $j$ than object $k$ is to object $l$), and when all triplet comparisons are available (object $i$ is more similar to object $j$ than it is to object $k$). Note that the all-quadruple comparisons setting corresponds to complete ordinal information. We focus on the setting of ordinal triplet comparisons, also referred to as the method of triads \cite{torgerson1952multidimensional}, as it appears to be more natural in practice, such as when working with human subjects \cite{miller1956}. 

\subsection{Discrete setting}
\label{sec:mds_discrete}
In practice, when dealing with a finite set of items that need to be embedded on the unit sphere, the problem is as follows: Given a set of row ranks $(r_{ij}: i,j \in [n])$ with each $(r_{i1}, \dots , r_{in})$ a permutation of $(1, \dots, n)$, and a dimension $p \geq 2$,
\begin{align}
\label{eq:mds}
\text{Find } x_1, \dots, x_n \in \bbS^{p-1} \text{ such that  } \langle x_i, x_j \rangle < \langle x_i, x_k \rangle \text{ whenever } r_{ij} < r_{ik}.
\end{align}

In the realizable setting, there are points $\hat{x}_1, \dots, \hat{x}_n \in \bbS^{p-1}$ such that $r_{ij} < r_{ik}$ whenever $  \langle \hat{x}_i, \hat{x}_j \rangle < \langle \hat{x}_i, \hat{x}_k \rangle$. Clearly, any orthogonal transformation of a solution to \eqref{eq:mds} is also a solution. Thus, even in the realizable setting, we can only hope to recover the original configuration up to an orthogonal transformation if we are only given access to ordinal comparisons. Of course, for finite $n$, the set of all realizing configurations is open (in $\bbR^{np}$). The question of uniqueness is then: {\em Is there enough information in the ordinal constraints to recover the original points up to an orthogonal transformation in the large-sample limit?}

In the realizable case of the Euclidean setting where points are in $\bbR^p$, ordinal comparisons based on distances between points are enough to recover the original points up to a similarity transformation in the large-sample limit \cite{kleindessner2014,arias2017some, arias2025ordinal}. We note that the results in these papers require the underlying points be dense in some open subset of $\bbR^p$, and thus do not immediately generalize to the case where the points are constrained to be on the sphere. 

\begin{remark}
As in the task of spherical external unfolding, in view of the fact that $ \|x-x'\|^2 = 2(1-\langle x,x' \rangle)$ for all $x,x' \in \bbS^{p-1}$, the problem is in fact the same as the corresponding problem in \cite{arias2025ordinal}, but with the embedding space being $\bbS^{p-1}$ instead of $\bbR^p$.
\end{remark}

\subsection{Continuum setting}
Again we consider the continuum setting, where the underlying set of items $\cX \subset \bbS^{p-1}$ will be taken to be uncountably infinite. In the discrete setting, two configurations in dimension $p$, $\{x_1, \dots, x_n \}$ and $\{x'_1, \dots, x'_n \}$ are indistinguishable if
\begin{align}
\label{eq:mds_equivalence}
    \langle x_i, x_j \rangle < \langle x_i, x_k \rangle \iff \langle x'_i, x'_j \rangle < \langle x'_i, x'_k \rangle \quad \text{for all } i,j,k \in [n].  
\end{align}

In the continuum setting, $\cX$ contains uncountably many objects, and so to extend the notion of indistinguishable configurations we follow \citet{shepard1966} and consider isotonic functions. A function $f:\cX \to \bbS^{p-1}$ is isotonic if it preserves the similarity ordering or preferences in the sense that 
 \begin{equation}
    \langle x,x'\rangle < \langle x,x''\rangle \iff \langle f(x),f(x')\rangle < \langle f(x),f(x'')\rangle, \quad \text{for all } x,x',x'' \in \cX.
\end{equation}

If $f$ is isotonic, then $\cX$ and $f(\cX)$ are indistinguishable configurations with respect to their ordinal similarity information. Clearly, an orthogonal transformation is isotonic on the sphere. We are interested in determining if there are conditions on $\cX$ under which every isotonic function is necessarily an orthogonal transformation. It turns out it is enough to consider weakly isotonic functions, where a weakly isotonic function $f: \cX \to \bbS^{p-1}$  (weakly) preserves similarity ordering in the sense that 
\begin{equation}
\label{eq:weakly_isotonic_sphere}
    \langle x,x'\rangle < \langle x,x''\rangle \implies \langle f(x),f(x')\rangle \leq \langle f(x),f(x'')\rangle, \quad \text{for all } x,x',x'' \in \cX.
\end{equation}

In the main result below, \thmref{mds_open}, we show that when $p \geq 3$ and $\cX$ has non-empty interior, then any non-constant, weakly isotonic function $f$ must act as an orthogonal transformation on $\cX$. When $p=2$, the same result requires the interior of $\cX$ contain a closed hemisphere --- and this requirement cannot be weakened. Indeed, consider the situation where $\cX$ is a closed hemisphere, e.g., $\cX = \{x \in \bbS^1 : \langle e_2, x \rangle \geq 0\}$. Then the function $f$ that maps each point to the point whose angle from $e_1$ is half of the original angle, i.e., $f(\cos\theta,\sin\theta) = (\cos(\theta/2),\sin(\theta/2))$, is weakly isotonic on $\cX$, but it does not coincide with an orthogonal transformation. 

We begin by establishing some key properties of a function $f$ when it is non-constant and weakly isotonic on an open ball: it must be continuous, it must be injective, and it must preserve midpoints. 
%Continuity and preservation of midpoints were also used in the Euclidean case \cite{kleindessner2014, arias2017some}, however the way we establish these properties and combine them is different, to account for the fact that the embedding space is the sphere. 
We use these properties to establish that the restriction of $f$ to an open ball must be an orthogonal transformation. Then, using a call to external unfolding relying on \thmref{external_open}, we show that $f$ is in fact an orthogonal transformation on all of $\cX$.

\begin{lemma}
\label{lem:continuity}
Suppose $\cX \subset \bbS^{p-1}$ is open. Then any weakly isotonic function on $\cX$ as in \eqref{eq:weakly_isotonic_sphere} must be continuous. 
\end{lemma}

\begin{proof}
\lemref{continuity} is an immediate consequence of \lemref{uniform_continuity}, stated and proved later.
\end{proof}

\begin{lemma}
\label{lem:equality}
Suppose $\cX \subset \bbS^{p-1}$ is open. Then any weakly isotonic function on $\cX$ as in \eqref{eq:weakly_isotonic_sphere} must satisfy
\begin{align}
\label{eq:equality}
\langle x,x'\rangle  = \langle x,x''\rangle \implies \langle f(x),f(x')\rangle  = \langle f(x),f(x'')\rangle, \quad \text{for all } x,x',x'' \in  \cX.
\end{align}
\end{lemma}

\begin{proof}
Take $x,x',x'' \in \cX$ such that $\langle x,x'\rangle  = \langle x,x''\rangle$. We assume them to be distinct, for otherwise the situation is trivial. Because $\cX$ is open and of how the inner product behaves on the sphere, it is not hard to see that we can built a sequence $(x_n) \subset \cX$ converging to $x''$ such that $\langle x,x'\rangle < \langle x,x_n\rangle$ for large enough $n$. 
(A natural choice is something like $x_n = \gamma_{(x,x'')}(1+1/n)$.)
Applying \eqref{eq:weakly_isotonic_sphere}, we get $\langle f(x),f(x')\rangle \leq \langle f(x),f(x_n)\rangle$ for large enough $n$, and passing to the limit using the continuity of the inner product and that of $f$ (\lemref{continuity}), we conclude that $\langle f(x),f(x')\rangle \leq \langle f(x),f(x'')\rangle$. And the reverse inequality holds by symmetry, allowing us to conclude.
\end{proof}

\begin{lemma}
\label{lem:injective}
Suppose $\cX \subset \bbS^{p-1}$ is open. Then any non-constant, weakly isotonic function $f:\cX \rightarrow \bbS^{p-1}$ in the sense of \eqref{eq:weakly_isotonic_sphere} must be injective.
\end{lemma}

\begin{proof} 
Consider such an $f$ that is not injective, and let $x_1, x_2 \in \cX$ distinct such that $f(x_1) = f(x_2)$. 

First, consider $x \in \cX$ such that $\langle x_1, x_2\rangle < \langle x_1, x\rangle$. By weak isotonicity, $1 = \langle f(x_1), f(x_2) \rangle  \leq \langle f(x_1), f(x) \rangle,$ implying that $f(x) = f(x_1)$. Similarly, if $\langle x_1, x_2\rangle = \langle x_1, x\rangle$, by \lemref{equality}, we can conclude that $f(x) = f(x_1)$. 

Next, consider $x \in \cX$ such that $\langle x_1, x\rangle < \langle x_1, x_2\rangle$. By convexity, $\gamma_{(x_1,x)}(t) \in \cX$ for any $t \in [0,1]$. Also, by the fact that $t \mapsto \langle x_1, \gamma_{(x_1,x)}(t) \rangle$ decreases over $[0,1]$ from $1$ to $\langle x_1, x\rangle$, there exists a $t_0 \in [0,1]$ such that  such that $\langle x_1, \gamma_{(x_1,x)}(t_0) \rangle =  \langle x_1, x_2 \rangle$. Define $z_k = \gamma_{(x_1,x)}(kt_0)$ and let $m$ be the greatest integer such that $m t_0 \leq 1$. We have thus constructed a sequence in $(z_k)_{k=1}^m\in \cX$ such that 
$$
\langle x_1, x_2 \rangle =  \langle x_1, z_1 \rangle = \langle z_1, z_2 \rangle = \dots  = \langle z_{m-1}, z_{m} \rangle < \langle z_m, x\rangle.
$$
Therefore, using \lemref{equality},
$$
1 = \langle f(x_1), f(x_2) \rangle =  \langle f(x_1), f(z_1) \rangle = \langle f(z_1), f(z_2) \rangle = \dots  = \langle f(z_{m-1}), f(z_{m}) \rangle \le \langle f(z_m), f(x)\rangle,
$$
implying that $f(x) = f(x_1)$. 

Thus, we have shown $f$ is constant.
\end{proof}

\begin{lemma}
\label{lem:midpoints}
     Suppose $\cX  = \sball(x_0,r)$ is an open ball in the sphere, and consider a non-constant weakly isotonic function $f:\cX \to \bbS^{p-1}$ in the sense of \eqref{eq:weakly_isotonic_sphere}. Then, $f$ preserves midpoints on $\sball(x_0,r/4)$.   
\end{lemma}

\begin{proof}
%We will show that $f$ preserves midpoints on $\sball(x_0,r/4)$.  
Fix $x,x' \in \sball(x_0,r/4)$ distinct and not diametrically opposed. Their midpoint $\mu := \gamma_{(x,x')}(1/2)$ is in $\sball(x_0,r/4)$ by convexity. Let $v_1, \dots, v_{p-2} \in \bbS^{p-1}$ be mutually orthogonal unit vectors, that are also orthogonal to $x$ and $x'$. Define $z_{0}^{+} = x$, $z_{0}^{-} = x'$, and $z_{i}^{\pm} = \langle x, \mu \rangle \mu \pm \sqrt{1 - \langle x, \mu \rangle^{2}}\, v_{i}$. By the triangle inequality,  $$
\| z_{i}^\pm - x_0 \| \leq \|z_i^\pm - \mu\| + \| \mu - x_0\| = \|x - \mu\| + \| \mu - x_0\| < r, 
$$ implying that $z_i^\pm \in \cX$. We also have 
$\langle  \mu, z_i^{\pm} \rangle =  \langle  \mu, x \rangle$ for $0 \leq i \leq p-2$, so that  by \lemref{equality}, $$
\langle f(\mu) , f(z_{0}^{+})\rangle = \langle f(\mu), f(z_{0}^{-}) \rangle = \dots = \langle f(\mu) , f(z_{p-2}^{+})\rangle = \langle f(\mu), f(z_{p-2}^{-}) \rangle.
$$
Without loss of generality, we take $f(\mu) = e_1$ and $f(z_i^\pm) = (\cos(\theta), \sin(\theta)u_i^\pm)$
 for some $\theta >0$ and $u_i^{\pm} \in \bbS^{p-2}$. Note that $\theta >0$ follows from the assumption that $f$ is non-constant, which implies by \lemref{injective} that $f$ is injective. 
 
 We also have
\begin{equation}
    \langle z_i^s, z_j^t\rangle = \langle z_k^u, z_l^v\rangle \quad \text{for all } i \neq j,\ k \neq l,\ \text{and } s,t,u,v \in \{+,-\},
\end{equation}
so that, by another application of \lemref{equality}, \begin{equation}
    \langle f(z_i^s), f(z_j^t)\rangle = \langle f(z_k^u), f(z_l^v)\rangle \quad \text{for all } i \neq j,\ k \neq l,\ \text{and } s,t,u,v \in \{+,-\}.
\end{equation}
We deduce that  $\langle u_i^s, u_j^t\rangle = \langle u_k^u, u_l^v\rangle$ for all $i \neq j,\ k \neq l$ and $s,t,u,v \in \{+,-\}$.  Then, by \lemref{z_config}, we conclude that, up to orthogonal transformation, $u_i^\pm$ are the basis vectors of $\bbR^{p-1}$. Thus, up to an orthogonal transformation, we have shown that 
\begin{equation}
\label{eq:midpoint_helper}
    f(z_i^\pm) = (\cos(\theta), \pm \sin(\theta)e_{i+2}).
\end{equation}

In particular, $f(x) = \cos(\theta)e_1 + \sin(\theta)e_2$ and $f(x') = \cos(\theta)e_1 -\sin(\theta)e_2$. As long as $\theta > \pi/2$, then $f(\mu)= e_1$ is indeed the midpoint of $f(x)$ and $f(x')$. But $\langle x, \mu \rangle > \langle x, z_1^+ \rangle$, so that by weak isotonicity $\langle f(x), f(\mu) \rangle \geq \langle f(x), f(z_1^+) \rangle$. Then, using \eqref{eq:midpoint_helper}, we find $\cos(\theta) \geq \cos^2(\theta)$, so that $\theta \geq \pi/2$. Finally, if $\theta = \pi/2$ then $f(x') = -f(x)$. Taking any  $x'' \in \cX$ such that $\langle x, x'' \rangle < \langle x, x' \rangle $, we obtain $\langle f(x),f(x'') \rangle \leq \langle f(x), f(x') \rangle  = -1$ so that $f(x'') =  -f(x) = f(x')$. This violates injectivity of $f$ and we conclude that $\theta > \pi/2$.
\end{proof}

\begin{lemma}
 \label{lem:mds_extension} 
Suppose $\cX \subseteq \bbS^{p-1}$ has non-empty interior, and consider a weakly isotonic function $f:\cX \to \bbS^{p-1}$ in the sense of \eqref{eq:weakly_isotonic_sphere}. Furthermore, assume that $f$ coincides with an orthogonal transformation on an open ball $\cB$ contained in $\cX$; if $p=2$, assume that $\cB$ is an open hemisphere. Then $f$ must coincide with that orthogonal transformation on the whole of $\cX$.
\end{lemma}

\begin{proof}
%The arguments are very similar to those in the Euclidean case \citep{arias2025ordinal}. By assumption, $f$ coincides with an orthogonal transformation on $\cB$. 
Without loss of generality, we may assume $f(x) = x$ for all $x \in \cB$. Now, take any $x \in \cX$. Then, by \eqref{eq:weakly_isotonic_sphere}, 
$$
\langle x, x'\rangle < \langle x, x''\rangle  \implies \langle f(x),x'\rangle \leq \langle f(x),x''\rangle, \quad \text{for all } x',x'' \in \mathcal{B},
$$
meaning that $x$ and $f(x)$ are weakly equivalent with respect to $\cB$ in the sense of \eqref{eq:ext_weakly_equiv}, and so are equal by \thmref{external_halfsphere} in the case $p=2$ and \thmref{external_open} in the case $p \geq 3$. We have just shown that $f(x) = x$ for any $x \in \cX$, allowing us to conclude. 
\end{proof}

\begin{theorem}
\label{thm:mds_open}
Suppose $\cX \subseteq \bbS^{p-1}$ has non-empty interior, and consider a non-constant weakly isotonic function $f:\cX \to \bbS^{p-1}$ in the sense of \eqref{eq:weakly_isotonic_sphere}.  
If $p=2$, additionally assume that the interior of $\cX$ contains a closed hemisphere. 
Then $f$ must coincide with an orthogonal transformation on $\cX$.
\end{theorem}

\begin{proof}
We reason by induction on $p$.
First, consider the case $p=2$. 
We may assume that $\cX$ contains in its interior the upper closed hemisphere, i.e., $\cH :=\{x : \langle x, e_2 \rangle \geq 0 \} \subset \sint(\cX)$, and that $f(e_1) = e_1$. By \lemref{mds_extension}, it is enough to show that $f$ acts as an orthogonal transformation on $\cH$. We first establish that $f(-e_1) = -f(e_1)$. Since $e_1 \in \cH \subset \sint(\cX)$, we may take $x,x' \in \cX$ symmetric about $e_1$ such that $\langle e_1, x \rangle = \langle e_1, x' \rangle$. Note that $-e_1 \in \cH \subset \sint(\cX)$ with $\langle -e_1, x \rangle = \langle -e_1, x' \rangle$.  
By \lemref{equality}, $\langle f(e_1), f(x) \rangle = \langle f(e_1), f(x') \rangle$ and $\langle f(-e_1), f(x) \rangle = \langle f(-e_1), f(x') \rangle$, so that $f(-e_1)= -f(e_1)$ by the fact that $f$ is injective. Additionally, since $\langle e_2, e_1 \rangle = \langle e_2, -e_1 \rangle$, it follows by \lemref{equality} that $\langle f(e_2), f(e_1) \rangle = \langle f(e_2), f(-e_1) \rangle = -\langle f(e_2), f(e_1) \rangle$, forcing $\langle f(e_2), f(e_1) \rangle = 0$. Consequently, up to an orthogonal transformation, we may assume that $f(e_2) = e_2$.

Let $\cU_0 = \{\pm e_1, e_2 \}$. So far we have shown that up to an orthogonal transformation, $f$ acts on $\cU_0$ as the identity.  Define $\cU_{t} \subset \cH$ recursively by letting $\cU_{t}$ be made of the midpoints of any pair of points in $\cU_{t-1}$ that are not diametrically opposed. We proceed to show that $f(x) = x$ for $x \in \cU_t$ by induction. Suppose we have shown that $f(x) = x$ for $x \in \cU_{t-1}$. Then for $x \in \cU_t$, there exist, by construction, $x',x'' \in \cU_{t-1}$ such that $\langle x, x' \rangle = \langle x, x'' \rangle$.  An application of \lemref{equality} together with the induction hypothesis forces $\langle f(x), x' \rangle = \langle f(x), x'' \rangle$. This confines $f(x)$ to two possibilities, $x$ and $-x$. However, we also have $\langle e_2, x \rangle > \langle e_2, e_1\rangle$ so that $\langle e_2, f(x) \rangle \ge \langle f(e_2), f(e_1) \rangle = \langle e_2, e_1\rangle = 0$. The only possibility is therefore $f(x) = x$. Thus, by induction, we have shown that $f(x) = x$ for any $x \in  \bigcup_{t\geq 0}\cU_t$.  And since $f$ is continuous and $\bigcup_{t\geq 0}\cU_t$ is dense in $\cH$, we may conclude that $f(x) = x$ for all $x \in \cH$.

Now, consider the case $p\geq 3$ and assume that the statement of the theorem is correct at $p-1$.
We may assume that there is $r>0$ such that $\sball(e_1, r) \subset \cX$, and that $f(e_1) = e_1$.
Set $\cB := \sball(e_1, r/4)$. By \lemref{mds_extension}, it is enough to show that $f$ acts as an orthogonal transformation on $\cB$. Let  $\phi = \cos^{-1}(1- r^2/32)$, so that the boundary of $\cB$ is exactly the $(p-2)$-dimensional sphere $\cS_{\phi} :=  \left\{ (\cos(\phi),\sin(\phi)z): z \in \bbS^{p-2} \right\}.$ Decompose $f(x)$ as $f(x) = (f_1(x), f_2(x))$ where $f_1(x) \in \bbR$ and $f_2(x) \in \bbR^{p-1}$. Since $\langle e_1, x \rangle$ is constant for all $x \in  \cS_{\phi}$, by \lemref{equality}, $\langle f(e_1), f(x) \rangle$ is also constant for all $x \in  \cS_{\phi}$. For $x = (\cos(\phi),\sin(\phi)z) \in \cS_\phi$, write 
\begin{equation}
 \label{eq:f_on_S}
     f_1(x) = \cos(\theta), \quad f_2(x) = f_2(\cos(\phi), \sin(\phi) z ) = \sin(\theta) g_{\phi}(z),
\end{equation}
for some constant $\theta >0$ and some function $g_{\phi}: \bbS^{p-2} \rightarrow \bbS^{p-2}$. Note that $\theta >0$ is implied by the assumption that $f$ is non-constant and therefore injective. It can also be checked that $g_{\phi}: \bbS^{p-2} \rightarrow \bbS^{p-2}$ inherits the weak isotonicity property from $f$, so by the induction hypothesis, $g_{\phi}$ is an orthogonal transformation, which we may take to be the identity transformation up to an orthogonal transformation, i.e., $g_\phi(z) = z$ for all $z \in \bbS^{p-2}$. 
% \begin{equation}
% \label{eq:f_on_S2}
%     f_1(x) = \cos(\theta), \quad f_2(x) = f_2(\cos(\phi), \sin(\phi) z ) = \sin(\theta) z, \quad \forall x \in \cS_{\phi}.
%\end{equation}

% For any $z,z',z'' \in \bbS^{p-2}$ such that $\langle z,z' \rangle < \langle z,z'' \rangle $, letting $x = (\cos(\theta_0), \sin(\theta_0)z) \in \cS_{\theta_0}$ and likewise  defining $x',x''$, we have $\langle x,x' \rangle < \langle x,x'' \rangle$. By \eqref{eq:weakly_isotonic_sphere}, $\langle f(x),f(x') \rangle \leq \langle f(x),f(x'') \rangle$. Using the form in \eqref{eq:f_on_S}, we conclude that $\langle g_{\theta_0}(z),  g_{\theta_0}(z') \rangle \leq  \langle g_{\theta_0}(z),  g_{\theta_0}(z'') \rangle$. We have therefore shown that $g: \bbS^{p-2} \rightarrow \bbS^{p-2}$ is weakly isotonic.

Next, we argue that $\theta$ must in fact be $\phi$. Take any $x \in \cS_{\phi}$. Because $\cB$ is contained in a hemisphere, there exists an $x' \in \cS_{\phi}$ such that $\langle x, x' \rangle = \langle x, e_1 \rangle$. If $x = (\cos(\phi),\sin(\phi)z)$ and $x' = (\cos(\phi),\sin(\phi)z')$, it follows that  $\langle f(x), f(x') \rangle = \langle f(x), e_1 \rangle$ and
$$\cos(\theta) = \cos^2(\theta) + \sin^2(\theta) \langle z, z' \rangle = \cos^2(\theta) (1-\langle z, z' \rangle) + \langle z, z' \rangle.$$ 
This is a quadratic equation in $\cos(\theta)$ with solutions  $\cos(\theta) = 1$ and $\cos(\theta) = \langle z , z' \rangle / (1- \langle z , z' \rangle)$, and by injectivity, the latter must hold. The fact that $\langle x, x' \rangle = \langle x, e_1 \rangle$ similarly implies that
$$\cos(\phi) = \cos^2(\phi) + \sin^2(\phi) \langle z, z' \rangle= \cos^2(\phi) (1-\langle z, z' \rangle) + \langle z, z' \rangle,$$
and the same arguments end in $\cos(\phi) = \langle z , z' \rangle / (1- \langle z , z' \rangle)$. We thus have that $\theta = \phi$, establishing that, up to an orthogonal transformation, $f$ acts as the identity on $\cS_{\phi}$. 

To complete the proof, let $\cV_0 = \cS_{\phi} \cup \{e_1\}$ and define $\cV_t$ recursively by letting $\cV_t$ be made of the midpoints of any pair of points in $\cV_{t-1}$ that are not diametrically opposed. Letting $\cV = \bigcup_{t \geq 0} \cV_t$, we  apply \lemref{midpoints} recursively to obtain $f(x) = x $ for any $x \in \cV$. Since $\cV$ is dense in $\cB$ and $f$ is continuous, we conclude $f(x) = x$ for all $x \in \cB$.
\end{proof}

\subsection{Discrete asymptotic setting }
\label{sec:mds_discrete_asymptotic}
We return to the discrete setting of \secref{mds_discrete}. The following result establishes that in the realizable setting with points dense in an open subset of $\bbS^{p-1}$, the ordinal constraints contain enough information to uniquely determine the points, up to an orthogonal transformation. 

% \begin{theorem}
% \label{thm:discrete_mds}
% Let $(x_n)$ be a sequence of distinct points in $\bbS^{p-1}$ and $(x'_{ni})$ be a triangular array with distinct points in each row in $\bbS^{p-1}$. Assume that, for each $n$, $\{ x_1, \dots,x_n\}$  and $\{ x'_{n1}, \dots,x'_{nn}\}$ are indistinguishable configurations in the sense of \eqref{eq:mds_equivalence}. \lnew{In addition, if $p=2$, assume that $(x_n)$ is dense in an open subset of $\bbS^{p-1}$ that contains a closed hemisphere.} If $p \geq 3$, assume that $(x_n)$ is dense in an open subset of $\bbS^{p-1}$. For each $n$, let $f_n : \{x_1, \dots, x_n\} \to \{ x'_{n1}, \dots,x'_{nn}\}$ such that $f_n(x_i) = x'_{ni}$ for all $i \in  [n]$. 
% Then $(f_n)$ is sequentially compact for the pointwise convergence topology and each \lnew{non-constant} function where it accumulates coincides with an orthogonal transformation on $\cX_{\infty} := \{x_i : i \geq 1 \}$.
% \end{theorem}

\begin{theorem}
\label{thm:discrete_mds}
Let $(x_n)$ be a sequence of distinct points in $\bbS^{p-1}$ such that $(x_n)$ is dense in an open subset of $\bbS^{p-1}$. If $p=2$, additionally assume that the open subset contains a closed hemisphere. Let $(f_n)$ be a sequence of weakly isotonic functions $f_n: \{x_1, \dots, x_n\} \to \bbS^{p-1}$ satisfying \eqref{eq:weakly_isotonic_sphere}. Then $(f_n)$ is sequentially compact for the pointwise convergence topology and each non-constant function where it accumulates coincides with an orthogonal transformation on $\cX_{\infty} := \{x_i : i \geq 1 \}$.
\end{theorem}

\begin{proof}
The sequential compactness of $(f_n)$ for the pointwise convergence topology is a standard result that rests on Cantor's diagonal argument. 
%(see Lemma 2 of \cite{arias2017some}). 

We may therefore assume without loss of generality that $(f_n)$ converges pointwise, with $f$ denoting the limit, so that $f(x) = \lim_{n \to\infty} f_n(x)$ for all $x \in \cX_{\infty}$. Take $x,x',x'' \in \cX_{\infty}$ such that $\langle x,x' \rangle < \langle x, x'' \rangle$. Then $x,x',x'' \in \cX_n := \{x_1, \dots, x_n\}$ for all $n$ large enough, and when this is the case, $\langle f_n(x) , f_n(x') \rangle \leq \langle f_n(x) , f_n(x'') \rangle $ by the fact that $f_n$ is weakly isotonic on $\cX_n$. Taking the limit as $n\to\infty$, we obtain $\langle f(x) , f(x') \rangle \leq \langle f(x) , f(x'') \rangle $. We conclude that $f$ is weakly isotonic on $\cX_{\infty}$. 

If $p=2$, let $\cB  \subset \S$ be an open ball that contains a closed hemisphere where $\cX_{\infty}$ is dense, and if $p \geq 3$, let $\cB \subset \S$ be an open ball where $\cX_{\infty}$ is dense.
Then, by \lemref{uniform_continuity} below, $f$ is uniformly continuous on $\cX_{\infty} \cap \cB$, and therefore admits a unique continuous extension on $\cB$, which we still denote by $f$. This extension is clearly weakly isotonic on $\cX := \cB \cup \cX_{\infty}$. Moreover, by \lemref{injective} and the assumption that $f$ is non-constant, $f$ is injective on $\cB$. Therefore, by \thmref{mds_open}, $f$ coincides with an orthogonal transformation on $\cB$, and by \lemref{mds_extension}, $f$ coincides with an orthogonal transformation on $\cX$, and therefore also on $\cX_{\infty}$.
\end{proof}

\begin{lemma}
\label{lem:uniform_continuity}
Suppose that $\cX$ is dense in an open ball  $\cB \subset \bbS^{p-1}$. Then any weakly isotonic function $f: \cX \to \bbS^{p-1}$ in the sense of \eqref{eq:weakly_isotonic_sphere} must be uniformly continuous on $\cX \cap \cB$.
\end{lemma}

\begin{proof} 
Write $\cB = \sball(x_0, r)$ for some $x_0 \in \bbS^{p-1}$ and $r>0$. It is enough to consider the case where $r < \sqrt{2}$, so that $\cB$ does not contain diametrically opposed points. Let $\phi = \cos^{-1}(1-r^2/2)$.  Take $x \neq x'$ in $\cX \cap \cB$. Define $\theta = \cos^{-1}\langle x,x' \rangle $. It will be sufficient to consider the case where $0< \theta < \min(\pi/3, \phi^2/2\pi)$.

First, we construct a finite sequence $z_i$ in $\sball(x_0, r)$ such that there is an increasing distance between consecutive points. Let $z_0 = x$ and set $\eps = (\pi\theta  - 3 \theta^2)/2\phi $. For $j \geq 1$, define $\theta_j =  \theta + j\eps$ and $z_k = \tilde{\gamma}_{(x, x_0)}(\sum^k_{j=1}  \theta_j )$, where $\tilde\gamma_{(x, x_0)}$ is defined in \eqref{eq:great_circle} as the great circle passing through $x$ and $x_0$. Let $m \geq 1$ be the maximum such that $\sum_{j=1}^m  \theta_j < \phi $. Since $m$ satisfies $2m\theta + 3m^2 \eps \geq  \phi$, we have 
\begin{align}
\label{eq:m_lower_bound}
m \geq \min \left( \phi/(4\theta), \sqrt{\phi/(6 \eps)} \right) \geq  \frac{\phi}{4} \sqrt{\frac{3}{\pi \theta}} .
\end{align}
Also, since $m\theta \leq \phi$, we also have $\theta \leq \theta_j \leq\theta_m  = \theta + m\eps \leq (\pi - \theta)/2 $. Therefore, by a Taylor expansion, for $j = 1, \dots, m$,
\begin{align}
\label{eq:neighbor_z}
\langle z_j, z_{j-1} \rangle 
&= \cos(\theta_{j-1} + \eps) <  \cos(\theta_{j-1}) - \eps\sin(\theta) = \langle z_{j-1}, z_{j-2} \rangle  - \eps\sin(\theta). 
\end{align}
 
Let $x_0 = x$, as well as $\delta = \eps\sin(\theta)/6$. Since $\cX$ is dense in $\cB$, there exist $x_1, \dots, x_m \in \cX$ such that $\max_i \|z_i - x_i\| \leq \delta$.  Therefore, by the Cauchy-Schwarz inequality and \eqref{eq:neighbor_z}, for $j=1,\dots, m$,  \begin{align*}
\langle x_j, x_{j-1} \rangle \leq \langle z_j, z_{j-1} \rangle + 3 \delta < \langle z_{j-1}, z_{j-2} \rangle - 3\delta \leq \langle x_{j-1}, x_{j-2} \rangle.
\end{align*}
This implies that 
$$
\langle x, x'\rangle > \langle x,x_1 \rangle >   \langle x_1,x_2 \rangle >\cdots   > \langle x_{m-1},x_m \rangle,
$$
which in turn implies, by weak isotonicity of $f$ on $\cX$, that
\begin{align}
\label{eq:isotonicity_chain}
    \langle f(x), f(x') \rangle \geq \langle f(x_0), f(x_1) \rangle  \geq \langle f(x_1), f(x_2) \rangle \geq \cdots \geq \langle f(x_{m-1}), f(x_m) \rangle.
\end{align}

We also have, for any $i,j \in [m]$ such that $1 \leq i \leq j-2$ ,
\begin{align*}
\langle z_i, z_j\rangle = \cos\bigl(\textstyle\sum^j_{k=i+1}  \theta_k  \bigr)  \leq \cos(\theta_{j-1} + \theta_j) < \cos \theta_j - \eps \sin\theta = \langle z_j, z_{j-1}\rangle - 6\delta, 
\end{align*}
so that \begin{align*}
    \langle x_i, x_j\rangle &\leq \langle z_i, z_j\rangle + 3\delta <  \langle z_j, z_{j-1}\rangle - 3\delta \leq  \langle x_{j}, x_{j-1}\rangle,
\end{align*}
and therefore, by weak isotonicity of $f$ on $\cX$, 
$$\langle f(x_i), f(x_j)\rangle \leq \langle f(x_{j}), f(x_{j-1}) \rangle.$$ 

Combining this with \eqref{eq:isotonicity_chain}, we have that 
$$\langle f(x_i), f(x_j)\rangle \leq \langle f(x), f(x')\rangle, \quad \text{for all } i,j \in [m].$$ 
Thus, letting $\eta = \cos^{-1}\langle f(x), f(x')\rangle $, $\{f(x_1), \dots, f(x_m)\}$ forms a $\eta-$packing of $\bbS^{p-1}$. It is well-known that such a packing must have cardinality $\leq C_0 \eta^{-(p-1)}$ for some constant $C_0$ depending only on $p$, and hence,
\begin{align}
\label{eq:m_upper_bound}
    m \leq C_0  [\cos^{-1}\langle f(x), f(x')\rangle ]^{-(p-1)}.
\end{align}

Combining \eqref{eq:m_lower_bound} and \eqref{eq:m_upper_bound}, we obtain 
$$
 \cos^{-1}\langle f(x), f(x')\rangle \leq C_1 [\cos^{-1} \langle x, x' \rangle]^{1/(2p-2)},
$$
or, equivalently, 
$$
 \cos^{-1}\left( 1 - \frac{\|f(x) - f(x')\|^2}{2} \right)\leq C_1 \left[\cos^{-1} \left( 1 - \frac{\|x - x'\|^2}{2} \right)\right]^{1/(2p-2)},
$$
for some other constant $C_1$ that depends only on $p$ and $\phi$ (and $\phi$ is a function of $r$).

Finally, using the bounds $t^2/4 \leq \cos^{-1}(1-t^2/2) \leq 4\sqrt{t}$ when $0 \leq t \leq 2$, we have 
$$\|f(x) - f(x')\| \leq C_2 \|x - x'\|^{1/(8p-8)},$$
for some other constant $C_2$ that depends only on $p$ and $\phi$ (or $r$). We just established this for all $x,x' \in \cX \cap \cB$, and it obviously implies that $f$ is uniformly continuous on $\cX \cap \cB$.
\end{proof}

\section{Spherical internal unfolding}
\label{sec:internal}

In internal unfolding, also referred to as ``unfolding" in the psychometrics literature (``bipartite localization" elsewhere \cite{einav2023}), the dataset contains two classes of items and only measurements between pairs of items, one in each class. Typically, in psychometrics, the measurements are preferences of different individuals for a set of different objects. In the Euclidean setting, the task is to embed both the individuals and objects in Euclidean space based on such measurements. Note that unlike in the external unfolding setting, the objects are not already embedded, instead, their location also needs to be determined. 
In spherical internal unfolding, the task is to embed both the individuals and objects on the unit sphere. (In another variant, the objects may be confined to the unit sphere while the individuals are unconstrained \cite{borg1980, smacof2}.) We focus on conditional unfolding in the ordinal setting, in which preferences of an individual are compared only to preferences by the same individual, i.e., information of the following sort is available: ``user $i$ prefers object $j$ over object $k$''. This corresponds to the method of triads \cite{torgerson1952multidimensional, miller1956}. 

\subsection{Discrete setting}
In practical settings, the problem of embedding individuals and objects into the sphere is as follows: Given a set of row ranks $(r_{ij} : i \in [m], j \in [n])$, with each $(r_{i1}, \dots, r_{in})$ a permutation of $(1,\dots, n)$, and a dimension $p \geq 2$, \begin{align}
\label{eq:internal_unfolding_discrete}
\text{Find } x_1, \dots, x_m \in \bbS^{p-1} \text{ and } y_1, \dots, y_n \in \bbS^{p-1} \\ \text{ such that } \langle x_i, y_j \rangle < \langle x_i, y_k \rangle \text{ whenever } r_{ij} < r_{ik}.
\end{align}

In the realizable setting, there are points $\hat{x}_1 \dots, \hat{x}_m, \hat{y}_1, \dots, \hat{y}_n \in \bbS^{p-1}$ such that  $r_{ij} < r_{ik}$ whenever $\langle \hat{x}_i, \hat{y}_j \rangle < \langle \hat{x}_i, \hat{y}_k \rangle$. Given a solution to \eqref{eq:internal_unfolding_discrete}, any orthogonal transformation of the solution, applied to both the individuals and objects, is also a solution. The question of uniqueness is otherwise analogous: {\em Is there enough information in the data to recover the original points and objects up to an orthogonal transformation in the large-sample ($m \to \infty$ and $n\to\infty$) limit?}

In the Euclidean setting, under certain technical conditions on the set of objects and set of individuals, ordinal comparisons of each individual’s preferences for the objects is enough to recover both the  the individuals and the objects, up to a similarity transformation in the large $m$ and $n$ limit \citep{arias2025ordinal}. As in the previous settings, the assumptions for these results also require that individuals are dense in some open set of $\bbR^p$, so that the results do not immediately extend to the spherical setting.  

\begin{remark}
In another variant of spherical internal unfolding, one may wish to embed only the individuals on the sphere while embedding the objects in the ambient Euclidean space, or vice versa. The \texttt{unfolding} function in the \texttt{SMACOF} package \cite{smacof2} allows users to do that. We leave the question of uniqueness in this setting as an open question.
\end{remark}

\subsection{Continuum setting}
Again we consider the continuum limit, where the number of individuals and objects form uncountably infinite sets, $\cX$ and $\cY$.  In the discrete setting, we say that two configurations of individuals and objects, $\{x_1, \dots, x_m ; y_1, \dots, y_n \}$ and  $\{x'_1, \dots, x'_m ; y'_1, \dots, y'_n \}$, are indistinguishable if \begin{equation}
\label{eq:equiv_internal}
    \langle x_i, y_k \rangle < \langle x_i, y_l \rangle \iff   \langle x'_i, y'_k \rangle < \langle x'_i, y'_l \rangle \qquad \text{for all } i,k,l \in [m] \times [n] \times [n].
\end{equation}
 
When passing to the continuum, we extend this notion by considering pairs of injective functions, $f: \cX \to \bbS^{p-1}$ and $g: \mathcal{Y} \to \bbR^p$, that preserve preferences in the sense that,
\begin{equation}
\label{eq:internal1}
    \langle x, y \rangle < \langle x, y' \rangle \iff   \langle f(x), g(y) \rangle <\langle f(x), g(y') \rangle \qquad \text{for all } x,y,y' \in \cX \times \cY \times \cY.
\end{equation}

Clearly, if $f=g$ is an orthogonal transformation, then \eqref{eq:internal1} holds. We wish to determine conditions on $\cX$ and $\cY$ under which every pair of functions $(f,g)$  satisfying \eqref{eq:internal1} must be exactly like that. Following \cite{arias2025ordinal}, it turns out to be sufficient to consider a weaker variant where we consider pairs of functions,  $f: \cX \to \bbS^{p-1}$ and $g: \mathcal{Y} \to \bbS^{p-1}$, that preserve preferences in the sense that \begin{equation}
\label{eq:internal2}
    \langle x, y \rangle < \langle x, y' \rangle \implies   \langle f(x), g(y) \rangle \leq \langle f(x), g(y') \rangle \qquad \text{for all } x,y,y' \in \cX \times \cY \times \cY.
\end{equation}

\begin{remark}
While we are able to establish the spherical analogue of Proposition~4.3 in \cite{arias2025ordinal}, a result that makes assumptions on $f(\cX)$ and $g(\cY)$, we are not able to do the same for Theorem~4.5 there, which is in that paper the main result as it does away with such `unnatural' assumptions. For this reason, we make simplifying assumptions below that carry the benefit of being more transparent, and understood as applying to settings where the object set is embedded as the entire sphere.
\end{remark}

\begin{theorem}
\label{thm:internal_continuum_main}
Suppose $\cX \subset \bbS^{p-1}$ is open and that $\cY = \bbS^{p-1}$. 
If $p=2$, additionally assume that the interior of $\cX$ contains a closed hemisphere. 
Consider a pair of functions $(f,g)$ satisfying \eqref{eq:internal2} such that $g(\bbS^{p-1}) = \bbS^{p-1}$.
Then $f$ and $g$ coincide with the same orthogonal transformation on $\cX$ and $\cY$, respectively.
\end{theorem}

\begin{proof}
We first show that $f = g$ on $\cX$.
Assume, for the sake of contradiction, that there is some $x \in \cX$ such that $f(x) \neq g(x)$. By \eqref{eq:internal2}, $$
\langle f(x), g(y) \rangle \leq  \langle f(x), g(x) \rangle < 1, \quad \text{for all } y \in \bbS^{p-1},
$$
which contradicts that fact that $g(\bbS^{p-1}) = \bbS^{p-1}$. 

Now that we know that $f=g$ on $\cX$, by \eqref{eq:internal2}, we have 
$$
\langle x,x' \rangle < \langle x,x'' \rangle \implies \langle f(x),f(x') \rangle < \langle f(x),f(x'') \rangle \quad \text{for all }x,x',x'' \in \cX.
$$ 
Therefore, $f$ is weakly isotonic on $\cX$, implying by \thmref{mds_open} that $f$ coincides with an orthogonal transformation on $\cX$, which we take to be the identity without loss of generality.
At this point we thus have
$$f(x) = g(x) = x, \quad \text{for all } x \in \cX.$$ 

To complete the proof we must show that $g(y) = y$ for all $y \in \bbS^{p-1}$. 
Choose any ball $\cB := \sball(x_0,r) \subset \cX$ and set $\phi := \cos^{-1}(1-r^2/2)$. 
Define  $\cB_m := \{y \in \bbS^{p-1} : \cos^{-1}(\langle x_0, y \rangle) < m \phi\}$. We proceed by induction. Suppose we have established that $g(y) = y$ for $y \in \cB_m \cap \cY$ for some integer $1 \leq m \leq \pi/\phi$. This is true for $m=1$ since $\cB_1 = \cB$ and we have shown above that $f = g$ coincide with the identity on $\cB$. Now, take $y \in \cB_{m+1} \backslash \cB_m$. Let $v$ be the unit vector in the direction $y - \langle y, x_0 \rangle x_0$ and let $y(t) := \cos(t\phi) x_0 + \sin(t\phi)v$. 
It follows that $y(t) \in \cB_m$ if and only if $|t| < m$.

In particular, we consider $y(-m+1)$ and $y(- m + 1/2)$ both in $\cB_m$. Noting that $y = y(m+\delta)$  for some $0 \leq \delta < 1$, we obtain that the midpoint of $y(-m+1)$ and $y$ is $y(1/2 + \delta/2) \in \cB_1$ and the midpoint of $y(- m + 1/2)$ and $y$ is $y(1/4 + \delta/2) \in \cB_1$. 

The orthogonality of $x_0$ (as a vector) and $v$, together with $\langle x_0, y \rangle^2 + \langle v, y \rangle^2 =1$, implies that $y$ is the unique point $z \in \mathbb{S}^{p-1}$ satisfying
\begin{align}
\langle x_0, z \rangle = \langle x_0, y \rangle 
\quad\text{and}\quad
\langle v, z \rangle = \langle v, y \rangle.
\end{align}
%\begin{align}
%\begin{bmatrix}
%x_0^\top \\
% v^\top
%\end{bmatrix}\, z = \begin{bmatrix}
%\langle x_0, y \rangle \\
%\langle v, y \rangle
%\end{bmatrix},
%\qquad
%\text{subject to } z \in \mathbb{S}^{p-1},
%\end{align}
And, since $\text{span}\{x_0, v \} = \text{span}\{y(1/2+\delta/2), y(1/4 + \delta/2)\}$, $y$ is also the unique point $z \in \mathbb{S}^{p-1}$ satisfying 
\begin{align}
\label{eq:internal_system2}
\langle y(1/2+\delta/2), z \rangle = \langle y(1/2+\delta/2), y\rangle
\quad\text{and}\quad
\langle y(1/4 + \delta/2), z \rangle = \langle y(1/4 + \delta/2), y\rangle.
\end{align}
Thus, it suffices to show that $g(y)$ satisfies the above equations. We focus on the first one, that is, on showing that 
\[\langle y(1/2+\delta/2), g(y) \rangle = \langle y(1/2+\delta/2), y\rangle,\] 
establishing the second one is analogous. 
When $t \in (-m+1/2, -m+1)$ we have $\langle y(1/2+\delta/2), y \rangle >\langle y(1/2+\delta/2), y(t)\rangle $. Therefore by \eqref{eq:internal2} and using the fact that $g(y(t)) = y(t)$, it follows that $ \langle y(1/2+\delta/2), g(y) \rangle  \geq \langle y(1/2+\delta/2), y(t)\rangle $. Taking the limit as $t \nearrow -m+ 1$ gives 
\[\langle y(1/2+\delta/2), g(y) \rangle \geq \langle y(1/2+\delta/2), y(-m+1)\rangle  = \langle y(1/2+\delta/2), y\rangle.\] 
Similarly, when $t \in (-m+1, -m+3/2)$ we have $\langle y(1/2+\delta/2), y \rangle < \langle y(1/2+\delta/2), y(t)\rangle $. It follows in the same way that $ \langle y(1/2+\delta/2), g(y)  \rangle \leq \langle y(1/2+\delta/2), y(t) \rangle $, and taking the limit as $t \searrow -m+1$ gives 
\[\langle y(1/2+\delta/2), g(y)  \rangle \leq  \langle y(1/2+\delta/2), y(-m+1)  \rangle = \langle y(1/2+\delta/2), y\rangle.\] 
We conclude that $\langle y(1/2+\delta/2), g(y) \rangle = \langle y(1/2+\delta/2), y\rangle $.
\end{proof}

\subsection{Discrete asymptotic setting}

Returning to the discrete setting, the following result establishes a uniqueness result in the large sample limit of the realizable setting, under appropriate conditions on the limit individual and object subsets. 

\begin{theorem}
Let $(x_n)$ and $(y_n)$ be a sequences of distinct points on $\bbS^{p-1}$ and let $f_n: \{x_1, \dots, x_n\} \rightarrow \bbS^{p-1}$ and $g_n: \{y_1, \dots, y_n\} \rightarrow \bbS^{p-1}$ satisfy \eqref{eq:internal2} on $\cX_n := \{x_1, \dots, x_n\}$ and $\cY_n := \{y_1, \dots, y_n\}$. 
Then $(f_n, g_n)$ is sequentially compact for the pointwise convergence topology. 
Assume further that $\cX_\infty := \{x_n : n \ge 1\}$ is dense in an open subset of $\bbS^{p-1}$ and that $\cY_\infty := \{y_n : n \ge 1\}$ is dense in $\bbS^{p-1}$.
If $p=2$, assume in addition that $\cX_\infty$ is dense in an open set containing a closed hemisphere.  
Then, if $(f,g)$ denotes an accumulation point for $(f_n,g_n)$ such that $g(\cY_\infty)$ is dense in $\bbS^{p-1}$, $f$ and $g$ must coincide with the same orthogonal transformation on $\cX_{\infty}$ and $\cY_{\infty}$, respectively.
\end{theorem}

\begin{proof} 
The sequential compactness of $(f_n,g_n)$ results from Cantor's diagonal argument. 

Continuing, we may assume without loss of generality that $(f(x),g(y)) = \lim_{n \rightarrow \infty} (f_n(x), g_n(y))$ for all $x,y \in \cX_{\infty} \times \cY_{\infty}$.  Take $(x,y,y') \in \cX_{\infty} \times \cY_{\infty} \times \cY_{\infty}$  such that $\langle x, y \rangle <  \langle x, y' \rangle $. We then have $\langle f_n(x), g_n(y) \rangle \leq \langle  f_n(x), g_n(y') \rangle $, and taking the $n \to\infty$ limit yields $\langle f(x), g(y) \rangle \leq \langle  f(x), g(y') \rangle$. It is thus the case that
\begin{align}
\label{eq:internal_isotonic}
    \langle x,y \rangle < \langle x,y' \rangle \implies \langle f(x), g(y) \rangle \leq \langle  f(x), g(y') \rangle, \quad \text{for all }x,y,y' \in \cX_{\infty} \times \cY_{\infty} \times \cY_{\infty},
\end{align}
meaning that $(f,g)$ satisfies \eqref{eq:internal2} with respect to $\cX_\infty$ and $\cY_\infty$.
Note that this together with our assumptions that $\cY_\infty$ and $g(\cY_\infty)$ are both dense in $\bbS^{p-1}$ precludes $f$ from being constant.

Next, take $x \in \cX_{\infty}$ and $N_1 \subset \bbN$ such that $\lim_{n \in N_1} y_{n} = x$, the latter enabled by the fact that $\cY_\infty$ is dense in $\bbS^{p-1}$. We will show that $\ \lim_{n \in N_1}g(y_{n}) = f(x)$. Take $N_2 \subset \bbN$ such that $\lim_{n \in N_2}g(y_{n}) = f(x)$, enabled by the fact that $g(\cY_\infty)$ is dense in $\bbS^{p-1}$. Because the $y_{n}$, $n \in N_2$ are distinct, we may assume $\langle x, y_{n} \rangle <  1 $ for all $n \in N_2$ so that for any $n_2 \in N_2$, if $n_1 \in N_1$ is sufficiently large,  $\langle x, y_{n_2} \rangle < \langle x, y_{n_1} \rangle$. Therefore, by \eqref{eq:internal_isotonic}, $\langle f(x), g(y_{n_2}) \rangle \leq \langle f(x), g(y_{n_1}) \rangle$, and taking limits along $N_1$ and $N_2$ yields $\lim_{n_1 \in N_1} g(y_{n_1}) = f(x)$, since by construction $\lim_{n_2 \in N_2} g(y_{n_2}) = f(x)$.
We conclude that 
\begin{align}
\label{eq:f_g}
   \lim_{n \in N_1} y_{n} = x \implies  \lim_{n \in N_1} g(y_{n}) =f(x), \quad \text{for all } x \in \cX_{\infty}.
\end{align} We note that \eqref{eq:f_g}, together with a standard  compactness argument, implies the following:
\begin{equation}
\label{eq:f_g_alt}
\begin{gathered}
\text{For any $\eps > 0$, there is $\delta > 0$ such that,} \\ 
\text{if $x \in \cX_{\infty}$ and $y \in \cY_\infty$ satisfy $\|x - y\| \le \delta$, then $\|f(x) - g(y)\| \le \eps$.}
\end{gathered}
\end{equation} 

Next, take $x,x',x'' \in \cX_{\infty}$ such that $\langle x, x' \rangle < \langle x, x'' \rangle $. By the same fact that $\cY_\infty$ is dense in $\bbS^{p-1}$, there exists $N', N'' \subset \bbN$ such that $\lim_{n \in N'} y_n = x' $ and  $\lim_{n \in N''} y_n = x''$. For any large enough $n' \in N'$, there is $n'' \in N''$ such that $\langle x,y_{n'} \rangle < \langle x, y_{n''} \rangle $, and so by \eqref{eq:internal_isotonic}, $\langle f(x),g(y_{n'}) \rangle \leq \langle f(x), g(y_{n''}) \rangle $, and taking limits along $N'$ and $N''$, \eqref{eq:f_g} gives $\langle f(x),f(x')\rangle \leq \langle f(x), f(x'')\rangle. $
This being true for all $x,x',x'' \in \cX_{\infty}$ implies that $f$ is weakly isotonic on $\cX_{\infty}$. Applying \thmref{discrete_mds} with the item set $\cX_{\infty}$ and constant sequence $(f)$, we may conclude that $f$ coincides with an orthogonal transformation on $\cX_\infty$.  Without loss of generality, we assume henceforth that $f(x) = x$ for all $x \in \cX_\infty$. 

 In addition, by \eqref{eq:f_g_alt}, $g(y) = y$ for all $y \in \cl(\cX_\infty) \cap \cY_\infty$. To see this, fix $\eps > 0$ and let $\delta > 0$ be as in \eqref{eq:f_g_alt}. For any $y \in \cl(\cX_\infty) \cap \cY_\infty$, there is $x \in \cX_\infty$ such that $\|x-y\| \le \min\{\delta, \eps\}$. By \eqref{eq:f_g_alt}, it is also the case that $\|x-g(y)\| \le \eps$, and the triangle inequality then gives the bound $\|y-g(y)\| \le 2\eps$. Since $\eps>0$ was chosen arbitrary, we may indeed conclude that $g(y) = y$ for all $y \in \cl(\cX_\infty) \cap \cY_\infty$.

So far, we have established, by a straightforward extension, that $f = g = {\rm id}$ on $\cX := \cl(\cX_\infty)$ (using the assumption that $\cl(\cY_\infty)  = \bbS^{p-1}$).  From here the arguments used in the proof of \thmref{internal_continuum_main} can be adapted in a straightforward manner to further establish that $g$ can be extended to satisfy $g = {\rm id}$ on the whole sphere, allowing us to conclude.
\end{proof}

\subsection*{Acknowledgments} 
We would like to thank Finn Southerland for helpful discussions. 

\small
\bibliographystyle{chicago}
\bibliography{ref}

\end{document}

%% file: notation.tex
\def\S{\bbS^{p-1}}
\def\arc{\mathrm{arc}}
\def\hplane{\mathrm{H}}

\def\sint{\operatorname{int}}
\def\cl{\operatorname{cl}}
\def\sball{\mathrm{cap}}
\def\ball{\mathrm{ball}}

\newtheorem{theorem}{Theorem}[section]

\newtheorem{lemma}[theorem]{Lemma}

\theoremstyle{definition}

\theoremstyle{remark}
\newtheorem{remark}[theorem]{Remark}

\numberwithin{equation}{section}
\numberwithin{figure}{section}

\def\beq{\begin{equation}} % \setcounter{equation}{1}}
\def\eeq{\end{equation}}
\def\beqn{\begin{eqnarray*}}
\def\eeqn{\end{eqnarray*}}
\def\Bitem{\begin{itemize}\setlength{\itemsep}{.2in}}
\def\bitem{\begin{itemize}\setlength{\itemsep}{.05in}}
\def\eitem{\end{itemize}}
\def\Benum{\begin{enumerate}\setlength{\itemsep}{.2in}}
\def\benum{\begin{enumerate}\setlength{\itemsep}{.05in}}
\def\eenum{\end{enumerate}}
\def\bmult{\begin{multline*}}
\def\emult{\end{multline*}}
\def\bcenter{\begin{center}}
\def\ecenter{\end{center}}
\def\bframe{\begin{frame}}
\def\eframe{\end{frame}}

\newcommand{\thmref}[1]{Theorem~\ref{thm:#1}}

\newcommand{\lemref}[1]{Lemma~\ref{lem:#1}}
\newcommand{\secref}[1]{Section~\ref{sec:#1}}
\newcommand{\figref}[1]{Figure~\ref{fig:#1}}

\def\cB{\mathcal{B}}

\def\cH{\mathcal{H}}

\def\cS{\mathcal{S}}

\def\cU{\mathcal{U}}
\def\cV{\mathcal{V}}

\def\cX{\mathcal{X}}
\def\cY{\mathcal{Y}}

\def\bbN{\mathbb{N}}

\def\bbR{\mathbb{R}}
\def\bbS{\mathbb{S}}

\def\eps{\varepsilon}
\def\implies{\ \Rightarrow \ }
\def\iff{\ \Leftrightarrow \ }

\def\1{\mathbbm{1}}